\documentclass{amsart}
\usepackage[utf8]{inputenc}
\usepackage{cite}
\usepackage{amsfonts}
\usepackage{amssymb}
\usepackage{amsthm}
\usepackage{enumitem}
\usepackage{amsmath}
\usepackage{graphicx}
\usepackage{stmaryrd}
\usepackage{csquotes}
\usepackage{aliascnt}
\usepackage{mathtools}
\usepackage{bbm}
\usepackage[unicode,bookmarksopen]{hyperref}

\newtheorem{theorem}{Theorem}[section]

\newcommand{\MakeTheoremAndCounter}[2]{\newaliascnt{#1}{theorem}
\newtheorem{#1}[#1]{#2}
\aliascntresetthe{#1}
\expandafter\providecommand\csname#1autorefname\endcsname{#2}}

\MakeTheoremAndCounter{lemma}{Lemma}
\MakeTheoremAndCounter{claim}{Claim}
\MakeTheoremAndCounter{example}{Example}
\MakeTheoremAndCounter{fact}{Fact}
\MakeTheoremAndCounter{question}{Question}
\MakeTheoremAndCounter{corollary}{Corollary}
\MakeTheoremAndCounter{notation}{Notation}
\MakeTheoremAndCounter{problem}{Problem}
\MakeTheoremAndCounter{proposition}{Proposition}
\MakeTheoremAndCounter{conjecture}{Conjecture}

\theoremstyle{definition}
\MakeTheoremAndCounter{definition}{Definition}
\newtheorem*{remark}{Remark}

\newcommand{\pr}{\mathrm{pr}}

\newcommand{\supp}{\mathrm{supp}\,}
\newcommand{\dom}{\mathrm{dom}\,}
\newcommand{\graph}{\mathrm{graph}\,}
\newcommand{\intby}{~\mathrm{d}}

\newcommand{\titleofthepaper}{Low-complexity Haar null sets without $G_\delta$ hulls in $\mathbb{Z}^\omega$}
\begin{document}
\title{\titleofthepaper}

\author{Don\'at Nagy}
\address{E\"otv\"os Lor\'and University, Institute of Mathematics, P\'azm\'any P\'eter
s. 1/c, 1117 Budapest, Hungary}
\email{nagdon@bolyai.elte.hu}

\thanks{The author was supported by the National Research, Development and Innovation Office -- NKFIH, grant no.~104178.\\
\includegraphics[height=1cm]{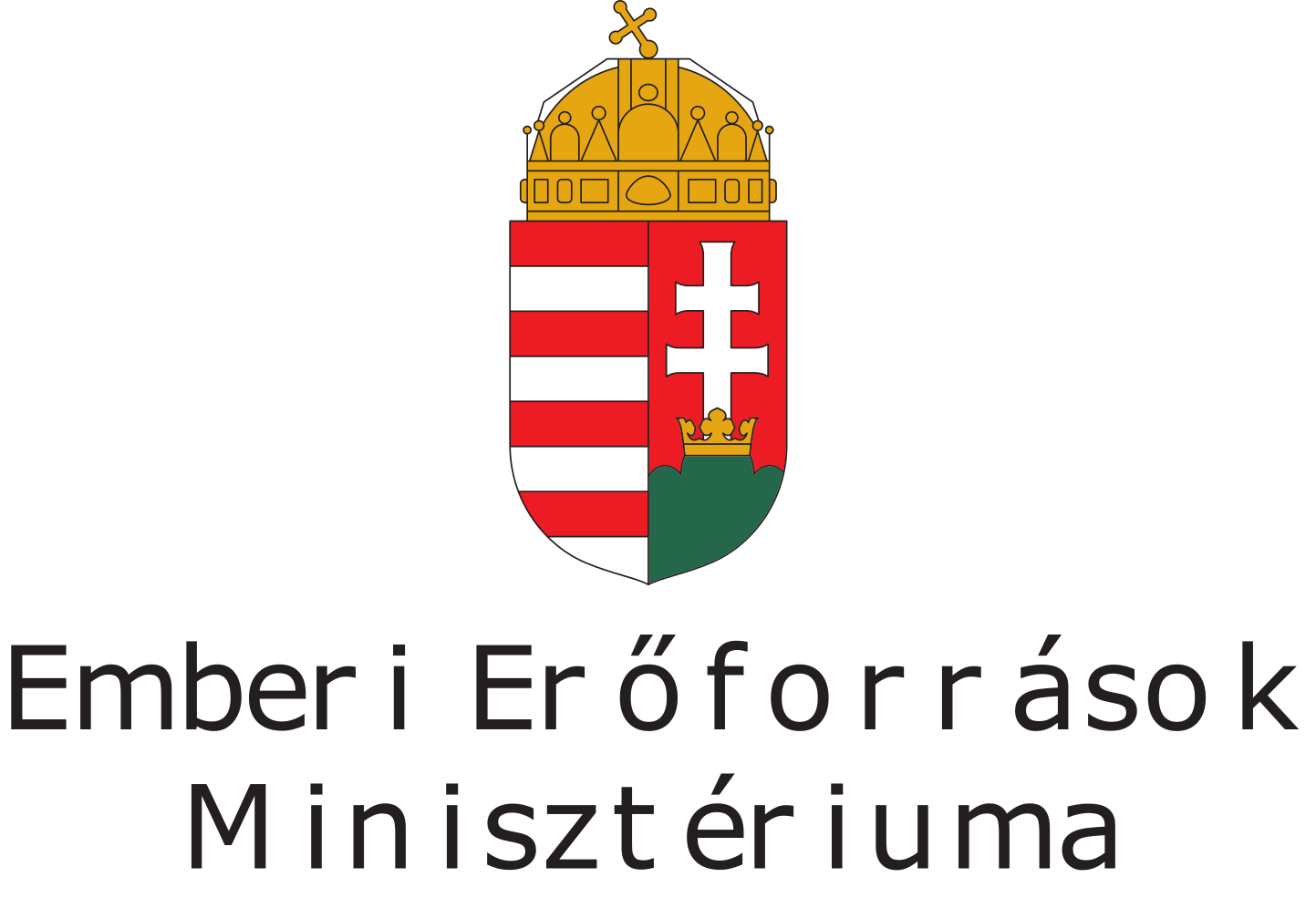} \raisebox{0.5cm}{\parbox[c]{11cm}{\scshape Supported by the \'UNKP-17-3 New National Excellence Program of the Ministry of Human Capacities.}}}
\begin{abstract}
We show that for every $2\le \xi<\omega_1$ there exists a Haar null set in $\mathbb{Z}^\omega$ that is the difference of two $\mathbf{\Pi}^0_\xi$ sets but not contained in any $\mathbf{\Pi}^0_\xi$ Haar null set. In particular, there exists a Haar null set in $\mathbb{Z}^\omega$ that is the difference of two $G_\delta$ sets but not contained in any $G_\delta$ Haar null set. This partially answers a question of M. Elekes and Z. Vidny\'anszky. To prove this, we also prove a theorem which characterizes the Haar null subsets of $\mathbb{Z}^\omega$.
\end{abstract}

\maketitle

\tableofcontents

\section{Introduction}

It is well known that if $G$ is a locally compact (Hausdorff) group, then there exists a left Haar measure on $G$, that is, a regular left invariant Borel measure that is finite for compact sets and positive for non-empty open sets (see e.g.\ \cite[\S15]{HR}). This measure is unique up to a multiplicative constant and is a vital tool in studying locally compact groups. Unfortunately, it can be proved that groups that are not locally compact do not admit a measure with these useful properties. However, the notion of a set of Haar measure zero has a well-behaved generalisation that works in Polish (i.e.\ separable and completely metrizable), not necessarily locally compact groups. This notion of a \emph{Haar null set} was introduced by Christensen in \cite{Chr}.

We will use the following definition of Haar null sets (this is the most common definition in recent papers and differs slightly from the original definition of Christensen):

\begin{definition}
If $G$ is a Polish group, a set $A\subseteq G$ is called \emph{Haar null} if there exist a Borel set $B\supseteq A$ and a Borel probability measure $\mu$ on $G$ such that $\mu(gBh) = 0$ for every $g, h\in G$. A measure $\mu$ satisfying this is called a \emph{witness measure} (for the set $A$).
\end{definition}

\begin{fact}
Haar null sets form a translation-invariant $\sigma$-ideal.
\end{fact}

The difficult part of this result is that Haar null sets are closed under countable unions. In this paper we only work in abelian Polish groups, where this was proved by Christensen in \cite[Theorem 1]{Chr}; for a proof in the general case see \cite{Myc} and \cite{CK}. As we should expect, in a locally compact Polish group the Haar null sets and the sets of Haar measure zero coincide.

This notion of smallness is widely used in various areas of mathematics; for a recent survey about the properties and applications of Haar null sets see e.g.\ \cite{EN}.

It is well known that the Haar measures are regular, i.e.\ if $G$ is a locally compact Polish group, $\mu$ is a left or right Haar measure on $G$ and $A \subseteq G$ is $\mu$-measurable, then
\[\mu(A) = \inf \{\mu(U) : A \subseteq U, U\text{ is open}\}.\]
(A proof of this can be found in \cite[15.8]{HR}.) This immediately implies that if $A\subseteq G$ is $\mu$-measurable, then there exists a $G_\delta$ set $A'\supseteq A$ such that $\mu(A') = \mu(A)$, in particular if $A$ is a set of Haar measure zero, then it is contained in a $G_\delta$ set of Haar measure zero.

This naturally inspires the question \cite[$P_1$]{Myc}:
\begin{question}[Mycielski]
Suppose that $G$ is a Polish group and $Y\subset G$ is Haar null. Does there exist a $G_\delta$ Haar null set including $Y$?
\end{question}

In the case of abelian Polish groups this was answered by \cite[Theorem 1.3]{EVGd}:
\begin{theorem}[Elekes-Vidny\'anszky]\label{the:ev}
If $G$ is a non-locally-compact abelian Polish group then there exists a Borel Haar null set $B\subset G$ that cannot be covered by a $G_\delta$ Haar null set.
\end{theorem}

As they used methods which construct Borel sets without giving an upper bound on the Borel class, they left the following questions open:
\begin{question}[Elekes-Vidny\'anszky]\label{que:ev1}
Let $G$ be a non-locally-compact abelian Polish group. Does there exist an $F_\sigma$ Haar null set that cannot be covered by a $G_\delta$ Haar null set?
\end{question}

Moreover, referring to their example of a set that cannot be covered by a $G_\delta$ Haar null set they asked the following:

\begin{question}[Elekes-Vidny\'anszky]\label{que:ev2}
What is the least complexity of such a set? And in general, what is the
least complexity of a Haar null set that cannot be covered by a $\mathbf{\Pi}^0_\xi$ Haar null set?
\end{question}

We partially answer this second question by proving the following result:

\newcommand{\thmmain}[2]{
\begin{theorem}#1
In the (non-locally-compact abelian Polish) group $\mathbb{Z}^\omega$ for every $2\le \xi<\omega_1$ there exists a Haar null set#2 that is the difference of two $\mathbf{\Pi}^0_\xi$ sets but is not contained in any $\mathbf{\Pi}^0_\xi$ Haar null set.
\end{theorem}
}\thmmain{}{}

\begin{remark} Studying this set of problems is also motivated by the well-known question \cite{DaOP} of Darji asking whether every uncountable Polish group can be written as a union of a meager and a Haar null set (which turns out to be equivalent to a strong variant of our problem), and also motivated by the similar result \cite{Ba} of Banakh in the non-abelian case. These connections are described in \autoref{sec:last}.
\end{remark}

\section{Preliminaries}

As usual, $\mathbb{N}$ and $\omega$ will both denote the set of nonnegative integers. We will use \enquote{$\mathbb{N}$} when we use this set as a topological space (with the discrete topology) and use \enquote{$\omega$} when we use it as an ordinal or index set. We use the notation $\llbracket m, n\rrbracket=[m,n]\cap \mathbb{Z}$ for sets consisting of consecutive integers, and as usual, $\mathbb{Z}_+$ denotes the set of positive integers. For $n\in\omega$ let $\pr_n$ be the canonical projection $\pr_n : \mathbb{Z}^\omega \to \mathbb{Z}, a \mapsto a(n)$.

We will use some notation related to sequences (i.e.\ functions $s$ whose domain is either a natural number or $\omega$). Let $S$ be an arbitrary set. Let $S^{<\omega}=\bigcup_{n\in\omega}S^n$ be the set of finite sequences of elements of $S$. For $s\in S^{<\omega}$, $|s|$ denotes the length of $s$. For a sequence $s\in S^{<\omega}$, let $[s]\subseteq S^{\omega}$ be the set of sequences which have $s$ as an initial segment, i.e.
\[[s]=\{x\in S^\omega : x{\upharpoonright_{\dom(s)}}=s\}.\]



As usual, $\mathcal{B}(X)$ denotes the Borel sets of a space $X$.

\section{Haar null sets in \texorpdfstring{$\mathbb{Z}^\omega$}{Z to the omega}}
Let $\varrho_k$ be the uniform probability measure on $\llbracket 0, k\rrbracket$, that is, the (Borel probability) measure on $\mathbb{Z}$ defined by $\varrho_k(X)=\frac{|X\cap \llbracket 0, k\rrbracket|}{k+1}$.

If $(a(n))_{n\in\omega}$ is a sequence of positive integers, then let $\mu_a$ be the Borel probability measure on $\mathbb{Z}^\omega$ defined as the product $\bigotimes_{n\in\omega} \varrho_{a(n)}$. Clearly $\supp\mu_a = \prod_{n\in\omega} \llbracket 0, a(n)\rrbracket$.

The following result characterizes the Borel Haar null subsets of $\mathbb{Z}^\omega$:
\begin{theorem}\label{lem:borelsubsets}
A Borel subset $B \subseteq \mathbb{Z}^\omega$ is Haar null if and only if there exists a sequence of positive integers $(a(n))_{n\in\omega}$ such that $\mu_a(B+x)=0$ for every $x\in \mathbb{Z}^\omega$.
\end{theorem}
\begin{definition}
 We call a sequence $(a(n))_{n\in\omega}$ satisfying this a \emph{witness sequence} for $B$.
\end{definition}

This theorem is motivated by \cite[Theorem 4.1]{So}, but that result works in a more general setting and shows that one can always choose a witness measure from another class of \enquote{simple measures}. 
It would be possible to modify the proof of \cite[Theorem 4.1]{So} to prove our result, but due to technical difficulties we give a different, self-contained proof.

\begin{proof}
The \enquote{if} part of the statement is trivial. To prove the \enquote{only if} part, assume that $B$ is a Borel Haar null subset of $\mathbb{Z}^\omega$.

It is not very hard to prove that every Haar null set has a witness measure with compact support, for a proof of this, see e.g.\ \cite[Theorem 4.1.4]{EN}. Using this, let $\mu$ be a witness measure for $B$ such that $\supp\mu$ is compact and therefore $\pr_n(\supp\mu)\subset \mathbb{Z}$ is compact (i.e.\ finite) for every $n\in\omega$.

To simplify the calculations, also suppose that $\supp\mu$ only contains sequences with nonpositive elements. This is always possible, as we may replace $\mu$ by $\mu'(X)=\mu(X+\ell)$ where $\ell\in \mathbb{Z}^\omega$ is the sequence $\ell(n)=\max(\pr_n(\supp\mu))$.

Let $M(n)= -\min(\pr_n(\supp\mu))$. It is clear that $\supp\mu\subseteq \prod_{n\in\omega} \llbracket -M(n), 0\rrbracket$. Choose a sequence $N(n)$ of (large) positive integers such that $N(n)>2M(n)$ (for every $n\in\omega$) and moreover
\[\prod_{n\in\omega}\left(1-\frac{M(n)}{N(n)+1}\right)>0.\]

Let $\nu$ be the measure $\nu=\mu * \mu_N$. This is the convolution of Borel probability measures, hence itself a Borel probability measure on $\mathbb{Z}^\omega$. Applying the definition of convolution, then using that $\mu$ is a witness measure we can see that for every $x\in \mathbb{Z}^\omega$
\[\nu(B+x)=\int_{\mathbb{Z}^\omega} \underbrace{\mu(B+x-y)}_0\intby\mu_N(y) = 0,\]
i.e.\ $\nu$ is also a witness measure for $B$. It is easy to see that \[\supp\nu\subseteq \prod_{n\in\omega}\llbracket -M(n), N(n)\rrbracket.\]

The measure $\nu$ is a \enquote{uniformized} variant of $\mu$, in fact if we ignore some \enquote{border zones} then the measure will be \enquote{uniform} on the \enquote{central zone} and this central zone will have positive measure. This will allow us to restrict $\nu$ to this central zone, normalize it and get a witness measure that is of the form $\mu_a$ for a witness sequence $a\in \mathbb{Z}_+^\omega$.

This heuristic statement can be formalized as the following claim:

\begin{claim}\label{cla:fst} Define $a(n)=N(n)-M(n)$. The set $\supp\mu_a=\prod_{n\in\omega}\llbracket 0, a(n)\rrbracket$ has positive $\nu$-measure. Moreover, \[\mu_a(X)=\frac{\nu(X\cap\supp\mu_a)}{\nu(\supp\mu_a)}\] holds for every $X\subseteq \mathbb{Z}^\omega$ (and these are defined for the same sets).
\end{claim}
\begin{proof}
For every Borel set $X\subseteq \supp\mu_a$ one can apply Fubini's theorem to get
\[\nu(X)=\int_{\mathbb{Z}^\omega} \mu_N(X-y)\intby\mu(y)=\int_{\supp\mu} \mu_N(X-y)\intby\mu(y).\]
If $y\in\supp\mu$ is arbitrary, then $-M(n)\le y(n)\le 0$ for every $n\in\omega$, hence \[\varrho_{N(n)}(W-y(n))=\varrho_{N(n)}(W)\] for every $W\subseteq\llbracket 0, a(n)\rrbracket$ (using that $a(n)=N(n)-M(n)$ and $\varrho_k$ is a uniform distribution). Moreover, notice that $\mu_N$ is the product of the measures $(\varrho_{N(n)})_{n\in\omega}$ and the measure which assigns $\mu_N(P-y)$ to the Borel set $P\subseteq\mathbb{Z}^\omega$ is the product of the measures which assign $\varrho_{N(n)}(W-y(n))$ to the set $W\subseteq\mathbb{Z}$. Thus the equality $\varrho_{N(n)}(W-y(n))=\varrho_{N(n)}(W)$, which holds for every $W\subseteq\llbracket 0, a(n)\rrbracket$, yields that these two product measures coincide when they are restricted to Borel subsets of $\prod_{n\in\omega}\llbracket 0, a(n)\rrbracket=\supp\mu_a$. If we apply this to a Borel subset $X\subseteq\supp\mu_a$, then we get $\mu_N(X-y)=\mu_N(X)$ for every $y\in \supp\mu$. This implies that
\begin{equation}\label{eqn:nuIsmuN}\nu(X)=\int_{\supp\mu}\mu_N(X)\intby\mu(y)=\mu_N(X).\end{equation}
For $n\in\omega$ and $W\subseteq \mathbb{Z}$ it is straightforward from the definitions that
\[\varrho_{a(n)}(W)=\int_{W} f_n(w) \intby\varrho_{N(n)}(w)\]
where $f_n : \mathbb{Z} \to \mathbb{R}$ is the density function defined as
\[f_n(w)=\begin{cases}\frac{N(n)+1}{a(n)+1}&\text{ if $w\in \llbracket 0, a(n)\rrbracket$,}\\0&\text{ if $w\in \mathbb{Z}\setminus\llbracket 0, a(n)\rrbracket$.}\end{cases}\]
If $X\subseteq \mathbb{Z}^\omega$ is a Borel set, then taking the product of these yields that
\[\mu_a(X) = \int_{X} f(x) \intby\mu_N(x)\]
where $\lambda = \prod_{n\in\omega} \frac{N(n)+1}{a(n)+1}$ and $f: \mathbb{Z}^\omega \to \mathbb{R}$ is the density function
\[ f(x) =\begin{cases}\lambda&\text{ if $x(n)\in \llbracket 0, a(n)\rrbracket$ for all $n\in\omega$,}\\0&\text{otherwise.}\end{cases} \]
Notice that $\lambda$ is trivially positive, but $\lambda<\infty$ because we assumed that
\[\prod_{n\in\omega}\left(1-\frac{M(n)}{N(n)+1}\right)>0.\]
In particular, for a Borel set $X\subseteq \supp\mu_a\,\left( = \prod_{n\in\omega}\llbracket 0, a(n)\rrbracket\right)$ this means that
\begin{equation}\label{eqn:lambdaIntr}\mu_a(X)=\lambda\cdot \mu_N(X)\end{equation}
Considering the special case when $X=\supp\mu_a$, we get
\[\lambda\cdot \mu_N(\supp\mu_a)=\mu_a(\supp\mu_a)=1,\]
\begin{equation}\label{eqn:muNOnSuppMua}\mu_N(\supp\mu_a)=\frac{1}{\lambda}>0.\end{equation}
Using \eqref{eqn:nuIsmuN}, \eqref{eqn:lambdaIntr} and \eqref{eqn:muNOnSuppMua} we can see that $\nu(X)=\mu_a(X)\cdot\nu(\supp\mu_a)$ for every $X\subseteq \supp\mu_a$. Here $\nu(\supp\mu_a)>0$ is implied by \eqref{eqn:nuIsmuN} and \eqref{eqn:muNOnSuppMua}. Therefore for an arbitrary $X\subseteq \mathbb{Z}^\omega$, $\mu_a(X)=\frac{\nu(X\cap \supp\mu_a)}{\nu(\supp\mu_a)}$, and this concludes the proof of the claim.
\end{proof}
As we return to proving \autoref{lem:borelsubsets}, we already know that for every $x\in\mathbb{Z}^\omega$, $\nu(B+x)=0$. This implies that for every $x\in\mathbb{Z}^\omega$, $\nu((B+x)\cap \supp \mu_a)=0$ and hence (using the recently proved \autoref{cla:fst}) $\mu_a(B+x)=0$. This means that $a$ is indeed a witness sequence.
\end{proof}

\section{A function with a surprisingly thick graph}

We say that a partial function $f: X\to Y$ is $\mathbf{\Sigma}^0_\xi$-measurable if the preimages of open subsets of $Y$ are $\mathbf{\Sigma}^0_\xi$ subsets of $X$.

The following theorem is the analogue of \cite[Theorem 3.1]{EVGd} and the proof technique is also similar.

\newcommand{\wnseqs}{\mathbb{Z}_+^\omega}
\begin{theorem}\label{thm:thickgraph}
If $2\le\xi<\omega_1$, then there exists a partial function $f: \wnseqs \times 2^\omega \rightarrowtail \mathbb{Z}^\omega$ which satisfies the following properties:
\begin{enumerate}[label=(\roman*)]
\item $f$ is $\mathbf{\Sigma}^0_\xi$-measurable and $\graph(f)$ is the difference of two $\mathbf{\Pi}^0_\xi$ subsets of $\wnseqs \times 2^\omega \times \mathbb{Z}^\omega$,
\item $\displaystyle (\forall (a,x)\in\wnseqs\times2^\omega) ( (a, x) \in \dom(f) \Rightarrow f(a, x) \in \supp\mu_a)$,
\item $\displaystyle (\forall a\in\wnseqs)(\forall S\in \mathbf{\Pi}^0_\xi(2^\omega \times \mathbb{Z}^{\omega}))\left(\graph(f_a) \subseteq S \Rightarrow (\exists x \in 2^\omega) (\mu_a(S_x)>0)\right)$
\end{enumerate}
(Here for $a\in\wnseqs$, $f_a$ denotes the partial function $f_a : 2^\omega \rightarrowtail \mathbb{Z}^\omega$, $f_a(x) = f(a, x)$.)
\end{theorem}

The proof will use a large section uniformization result by Holick\'y. We will use the following statement, which is an immediate corollary of the results in \cite{Ho}.
\begin{corollary}\label{thm:Holicky}
Assume that $X$ and $Y$ are Polish spaces, $2\le \alpha\le\omega_1$, and $\mu : X\times \mathcal{B}(Y)\to [0,1]$ satisfies
\begin{enumerate}[label=(\alph*),nosep]
\item $\mu(x,\cdot)$ is a Borel probability measure on $Y$ for every $x\in X$, and
\item $\{x\in X : \mu(x, H)>r\}$ is open for every open $H\subset Y$ and $r\in \mathbb{R}$.
\end{enumerate}
Assume that $A\in\mathbf{\Sigma}^0_\alpha(X\times Y)$ and define $P = \{p\in X : \mu(p, A_p)>0\}$.

Then there exists a partial function $f : X\rightarrowtail Y$ such that $f$ is $\mathbf{\Sigma}^0_{\alpha}$-measurable, $\dom(f)=P$ and $\graph(f)\subseteq A$. Moreover, this partial function also satisfies that $\graph(f)\in \mathbf{\Pi}^0_{\alpha}(P\times Y)$.
\end{corollary}

\begin{proof}[Proof of \autoref{thm:Holicky}.]
To prove that this holds for some $X$, $Y$, $\alpha$, $\mu$, and $A$, first apply \cite[Lemma 2.1]{Ho} (and the remarks preceding it) for $X$, $Y$, $\alpha$, $\mu$, $B=A$, and $\alpha_0=1$. The lemma yields that $P\in\mathbf{\Sigma}^0_\alpha(X)$, because $\alpha_0=1$ implies $\alpha^*=\alpha$ (here \enquote{$B$}, \enquote{$\alpha_0$}, and \enquote{$\alpha^*$} are introduced in the statement of the lemma).

This means we can apply \cite[Theorem 3.5]{Ho} for $X$, $Y$, $\alpha$, $\mu$, $B=A\cap (P\times Y)$, and $\alpha_0=1$ and choose $f=\xi$ (here \enquote{$B$}, \enquote{$\alpha_0$}, \enquote{$\alpha^*$}, and a function \enquote{$\xi$} are introduced in the statement of the theorem, the theorem states that $\xi$ has all the necessary properties; we use that $\alpha^*=\alpha$ holds again). 
\end{proof}

\begin{proof}[Proof of \autoref{thm:thickgraph}.]
Let $U \in \mathbf{\Sigma}^0_\xi(2^\omega\times 2^\omega\times \mathbb{Z}^\omega)$ be universal for the $\mathbf{\Sigma}^0_\xi$ subsets of $2^\omega\times \mathbb{Z}^\omega$, that is, for every $B\in \mathbf{\Sigma}^0_\xi(2^\omega \times \mathbb{Z}^\omega)$ there exists an $x\in2^\omega$ such that $U_x=B$ (for the existence of such a set see \cite[Theorem 22.3]{Ke}). The preimage of this set under the continuous map $(x,g)\mapsto (x,x,g)$ is the set
 \[U'=\{(x,g)\in 2^\omega\times \mathbb{Z}^\omega: (x,x,g)\in U\},\]
which is hence a $\mathbf{\Sigma}^0_\xi$ set. Later we will use that $U'_x=U_{x,x}$ for every $x\in 2^\omega$. 

Notice that the set
\[S=\{(a, x, g) \in \wnseqs \times 2^\omega\times \mathbb{Z}^\omega: g\in\supp\mu_a\}\] 
is equal to
 \[\{(a,x,  g) \in \wnseqs \times 2^\omega\times \mathbb{Z}^\omega : (\forall n\in \omega) (0\le g(n) \le a(n))\}\]
and hence is trivially closed.

We can combine these to define the $\mathbf{\Sigma}^0_\xi$ set
\[U''=(\wnseqs\times U')\cap S\quad \big(\ \subseteq \wnseqs\times 2^\omega\times \mathbb{Z}^\omega\ \big).\]

We will apply \autoref{thm:Holicky} for the Polish spaces $X=\wnseqs\times 2^\omega$ and $Y=\mathbb{Z}^\omega$, the map $\mu: \wnseqs\times2^\omega\times \mathcal{B}(\mathbb{Z}^\omega)\to[0,1]$ defined by $\mu((a,x), S) = \mu_a(S)$, the set $A=U''$ and $\alpha=\xi$. As in \autoref{thm:Holicky}, define $P = \{p\in \wnseqs\times 2^\omega : \mu(p, A_p)>0\}$. It is clear that condition (a) is satisfied.

\begin{claim} Condition (b) of \autoref{thm:Holicky} is also satisfied, that is, $\{(a,x)\in \wnseqs\times 2^\omega : \mu((a,x),H)>r\}$ is open for every open $H\subset \mathbb{Z}^\omega$ and $r\in \mathbb{R}$.
\end{claim}
\begin{proof}
$H$ can be written as $H=\bigcup_{i\in \omega} [s_i]$ for some sequences $s_i\in\mathbb{Z}^{<\omega}$ and we may also assume that this union is disjoint. Then $\mu((a,x),H)=\mu_a(H)=\sum_{i\in\omega} \mu_a([s_i])$.

Notice that for $K\in\omega$, the set $\{(a,x)\in\wnseqs\times 2^\omega : \sum_{i\in K} \mu_a([s_i])>r\}$ is open, as $\sum_{i\in K} \mu_a([s_i])$ depends only on the first $\max_{i\in K} |s_i|$ elements of the sequence $a$. This means that
\begin{align*}
\{(a,x)\in \wnseqs\times 2^\omega : \mu((a,x),H)>r\} &= \{(a,x)\in\wnseqs\times 2^\omega : \sum_{i\in \omega} \mu_a([s_i])>r\}\\
&=\bigcup_{K\in\omega}\{(a,x)\in\wnseqs\times 2^\omega : \sum_{i\in K} \mu_a([s_i])>r\}
\end{align*}
is a union of open sets and this proves our claim.
\end{proof}

The application of \autoref{thm:Holicky} proves the existence of a partial function $f : \wnseqs \times 2^\omega \rightarrowtail \mathbb{Z}^\omega$ such that $f$ is $\mathbf{\Sigma}^0_\xi$-measurable, $\dom(f)=P$, $\graph(f)\subseteq A$, and $\graph(f)\in \mathbf{\Pi}^0_\xi(P\times \mathbb{Z}^\omega)$. To finish the proof of \autoref{thm:thickgraph} we show that this $f$ satisfies Properties (i)--(iii).

As $P=\dom(f)=f^{-1}(\mathbb{Z}^\omega)$ is a $\mathbf{\Sigma}^0_\xi$ set, $\graph(f)\in \mathbf{\Pi}^0_\xi(P\times \mathbb{Z}^\omega)$ is clearly the difference of two $\mathbf{\Pi}^0_\xi$ subsets of $\wnseqs \times 2^\omega \times \mathbb{Z}^\omega$, concluding the proof of Property (i).

Property (ii) is clear because for all $(a,x)\in\dom(f)$, 
\[f(a,x)\in U''_{a,x}=U'_x\cap \supp \mu_a\subseteq \supp \mu_a.\]

To prove Property (iii), suppose to the contrary that there exists $a\in\wnseqs$ and $S\in \mathbf{\Pi}^0_\xi(2^\omega\times \mathbb{Z}^\omega)$ such that $\graph(f_a)\subseteq S$ but for all $x\in 2^\omega$, $\mu_a(S_x)=0$. The complement of $S$ is the $\mathbf{\Sigma}^0_\xi$ set $B=(2^\omega\times \mathbb{Z}^\omega)\setminus S$. By the universality of $U$, there is an $x^*\in 2^\omega$ such that $U_{x^*}=B$. We know that for every $x\in 2^\omega$, $\mu_a(B_x)=1-\mu_a(S_x)=1>0$, in particular $\mu_a(B_{x^*})=\mu_a(U_{x^*,x^*})=\mu_a(U'_{x^*})>0$. It is clear from the definitions that $U''_{a,x^*} = U'_{x^*}\cap \supp(\mu_a)$, and so $\mu_a(U''_{a, x^*})>0$.

Therefore $(a, x^*)\in P=\dom(f)$ and then $\graph(f)\subseteq U''$ yields that
\[f(a, x^*)\in U''_{a,x^*}\subseteq U'_{x^*}=U_{x^*, x^*}=B_{x^*}.\]
But we also supposed that $\graph(f_a)\subseteq S$, and this yields $f(a, x^*)\in S_{x^*}= \mathbb{Z}^\omega\setminus B_{x^*}$, a contradiction.

This proves that $f$ indeed satisfies the requirements of \autoref{thm:thickgraph}.
\end{proof}

\section{The main result}

Now we are ready to prove our main result, which is stated in the following theorem.

\thmmain{\label{thm:main}}{ $E$}

\begin{proof} First we will construct some simple functions which will be useful in our proof.

Let us define the function 
\[\theta : \mathbb{Z}_+\times \{0,1\}\times \mathbb{Z}\to \mathbb{Z}\,,\qquad\theta(n, b,z)=(n-1)(n+4)+b(n+2)+z.\]
Elementary calculations show that $\theta(n, 1,0)=\theta(n,0,0)+(n+2)$, $\theta(n+1, 0,0)=\theta(n,1,0)+(n+2)$ and hence if we restrict $\theta$ to the set
\[T=\{(n, b, z) : n\in\mathbb{Z}_+, b\in \{0,1\}, z\in \llbracket 0, n+1\rrbracket\},\]
then $\theta(T)=\mathbb{N}$ is the set of nonnegative integers, and the restricted function $\theta{\upharpoonright_T}$ is an order preserving bijection (when $T$ is ordered lexicographically and $\mathbb{N}$ has its usual ordering). Let $\iota : \mathbb{N}\to T$ be the inverse of this restriction.

We can let $\theta$ act elementwise on sequences of length $\omega$, that is, we can define \[t : \mathbb{Z}_+^\omega\times 2^\omega\times \mathbb{Z}^\omega\to \mathbb{Z}^\omega\,,\qquad (t(a, x, g))(k) = \theta(a(k), x(k), g(k))\text{ for all $k\in\omega$.}\] Later we will use the fact that if $a\in\mathbb{Z}_+^\omega$ and $x\in 2^\omega$ are fixed, then $t(a,x,g) = t(a,x,0)+g$, i.e.\ $g \mapsto t(a,x,g)$ is a translation.

Analogously, we may also let $\iota$ act elementwise on sequences of length $\omega$ to get a function $i : \mathbb{N}^\omega\to T^\omega$. It is clear that both $t$ and $i$ are continuous (in fact, Lipschitz).

With a slight abuse of notation let us identify $T^\omega$ and the set
\[\mathcal{T}=\{(a,x,g)\in \wnseqs\times2^\omega\times\mathbb{Z}^\omega: (\forall k)(g(k)\in \llbracket 0, a(k)+1\rrbracket)\}\]
($T^\omega$ contains sequences of triples, $\mathcal{T}$ contains triples of sequences, the natural map between them is a homeomorphism). As $\iota$ is the inverse of a restriction of $\theta$, the same holds for $i$ and $t$: every $(a,x,g)\in \mathcal{T}$ satisfies $i(t(a,x,g))=(a,x,g)$ and every $s\in \mathbb{N}^\omega$ satisfies $t(i(s))=s$.

Let $f$ be a partial function $f : \wnseqs\times 2^\omega \rightarrowtail \mathbb{Z}^\omega$ which satisfies the conditions of \autoref{thm:thickgraph}.

Now we are able to define $E$ as
\begin{align*}
E=t(\graph(f))&= \{t(a,x,g) :(a,x,g)\in\graph(f)\}=\\
&= \{t(a,x,f(a,x)) :(a,x)\in\dom(f)\}.
\end{align*}

\begin{claim}
$E$ is the difference of two $\mathbf{\Pi}^0_\xi$ subsets of $\mathbb{Z}^\omega$.
\end{claim}

\begin{proof}
If we apply first the definition of $\mathcal{T}$ and then Property (ii) of \autoref{thm:thickgraph}, then we get
\[\mathcal{T}\supseteq \{(a,x,g) \in\wnseqs\times2^\omega\times\mathbb{Z}^\omega: g\in \supp\mu_a\}\supseteq \graph(f).\]
This implies that $\graph(f)=i(t(\graph(f)))=i(E)$ and hence $E=i^{-1}(\graph(f))$. As Property (i) states that $\graph(f)\subseteq \mathcal{T}$ is the difference of two $\mathbf{\Pi}^0_\xi$ subset of $\wnseqs\times2^\omega\times\mathbb{Z}^\omega$, it is also the difference of two $\mathbf{\Pi}^0_\xi$ subsets of $\mathcal{T}$ (using that $\mathcal{T}$ is closed). This means that its preimage under the continuous function $i : \mathbb{N}^\omega\to \mathcal{T}$ is the difference of two $\mathbf{\Pi}^0_\xi$ subsets of $\mathbb{N}^\omega$. As $\mathbb{N}^\omega$ is a closed subset of $\mathbb{Z}^\omega$, this yields that $E=i^{-1}(\graph(f))$ is indeed the difference of two $\mathbf{\Pi}^0_\xi$ subsets of $\mathbb{Z}^\omega$.
\end{proof}

\begin{claim}
$E$ is Haar null.
\end{claim}
\begin{proof}
We will show that $a_0=(1,1,\ldots)$ is a witness sequence for this fact. (By the way, the corresponding witness measure, $\mu_{a_0}$ is just the usual coin-flip measure with $\supp(\mu_{a_0})=\{0,1\}^\omega\subset \mathbb{Z}^\omega$.) It is clearly sufficient to show that $|(E+r)\cap \{0,1\}^\omega|\le 1$ for all $r\in \mathbb{Z}^\omega$. This is equivalent to saying that if $e,e'\in E$ and $e\neq e'$, then $|e(k)-e'(k)|\ge 2$ for at least one $k\in\omega$.

Fix arbitrary $e, e'\in E$ with $e\neq e'$. By the definition of $E$ there are $a, a'\in\wnseqs$ and $x, x'\in2^\omega$ such that $e=t(a,x,f(a,x))$ and $e'=t(a',x', f(a',x'))$. As we assumed that $e\neq e'$, we can find a $k\in\omega$ where $(a(k), x(k))\neq (a'(k), x'(k))$. Without loss of generality, we may assume that $(a(k), x(k))<(a'(k), x'(k))$ lexicographically. By Property (ii) of \autoref{thm:thickgraph} we know that $f(a,x)\in \supp(\mu_a)=\prod_{k\in\omega}\llbracket 0, a(k)\rrbracket$, hence $0\le (f(a,x))(k)\le a(k)$ and analogously $0\le (f(a',x')(k))\le a'(k)$. Straightforward and elementary calculations (using these bounds and the definition of $t$ and $\theta$) show that $e(k)+2\le e'(k)$ both in the case when $a(k)<a'(k)$ and in the case when $a(k)=a'(k)$ (and hence $x(k)<x'(k)$, i.e.\ $x(k)=0$ and $x'(k)=1$). These allow us to conclude that $E$ is indeed Haar null.
\end{proof}

\begin{claim}
There is no Haar null set $H\in \mathbf{\Pi}^0_\xi(\mathbb{Z}^\omega)$ containing $E$.
\end{claim}
\begin{proof}
Suppose that $H \in \mathbf{\Pi}^0_\xi$ is such a set. By \autoref{lem:borelsubsets} there exists a witness sequence $a\in \mathbb{Z}_+^\omega$ such that $\mu_a(H+r)=0$ for every $r\in\mathbb{Z}^\omega$. As $t$ is continuous, the section map \[t_a : 2^\omega\times\mathbb{Z}^\omega\to \mathbb{Z}^\omega\,,\quad(x,g)\mapsto t(a, x,g)\] is also continuous, hence $S=t_a^{-1}(H)\subseteq 2^\omega\times \mathbb{Z}^\omega$ is a $\mathbf{\Pi}^0_\xi$ set.

It is easy to check that $\graph(f_a)\subseteq S$, and therefore by Property (iii) of \autoref{thm:thickgraph} there exists an $x^*\in 2^\omega$ such that $\mu_a(S_{x^*})>0$. By the definition of $S$, $t(a, x^*, S_{x^*})\subseteq t_a(S)\subseteq H$. But $g\mapsto t(a,x^*, g)$ is a translation, so a translate of $H$ contains $S_{x^*}$, but $S_{x^*}$ has positive $\mu_a$-measure and hence this contradicts that $a$ is a witness sequence for $H$.
\end{proof}
This concludes the proof of our main theorem.
\end{proof}

\section{Related results and questions}\label{sec:last}

  
\autoref{thm:main} implies that if $2\le\xi<\omega_1$, then there is a $\mathbf{\Delta}^0_{\xi+1}$ Haar null subset of $\mathbb{Z}^\omega$ that is not contained in any $\mathbf{\Pi}^0_\xi$ Haar null set. This means that the possible answers for \autoref{que:ev2} are narrowed down to two classes of the Borel hierarchy: the least possible complexitiy class of a Haar null set that is not contained in a $\mathbf{\Pi}^0_\xi$ Haar null set must be either $\mathbf{\Sigma}^0_\xi$ or $\mathbf{\Delta}^0_{\xi+1}$. (Note that our result showed more than this: even if we subdivide $\mathbf{\Delta}^0_{\xi+1}$ into the classes of the difference hierarchy, our example is in one of the lowest classes.)

Therefore in $\mathbb{Z}^\omega$ essentially only the following question is left open:
\begin{question}
For a given $2\le \xi<\omega_1$, is there a $\mathbf{\Sigma}^0_\xi$ Haar null set in $\mathbb{Z}^\omega$ that cannot be covered by a $\mathbf{\Pi}^0_\xi$ Haar null set?
\end{question}

The cases of other non-locally-compact abelian Polish groups are also left open. As $\mathbb{Z}^\omega$ is among the \enquote{nicest} non-locally-compact Polish groups, it is plausible that the answer will be similar in those other groups.

While the questions of Elekes and Vidny\'anszky are only about the abelian case, it is natural to generalize them for arbitrary Polish groups:

\begin{question}\label{que:general}
For a given Polish group $G$ and $2\le \xi<\omega_1$, what is the least complexity of a Haar null set in $G$ that cannot be covered by a $\mathbf{\Pi}^0_\xi$ Haar null set?
\end{question}

The paper \cite{Ba} examines this generalized question and proves the following result, answering \autoref{que:general} for a particular (non-abelian, Polish) group $H$ and $\xi=2$:
\begin{theorem}[Banakh]
There exists a Polish meta-abelian group $H$ containing a subgroup $F\subset H$ such that $F$ is a $F_\sigma$ Haar null set in $H$ but every $G_\delta$ set $G\subset H$ containing $B$ is thick and hence is not Haar null in $H$.
\end{theorem}
(A topological group $H$ is called \emph{meta-abelian} if it contains a closed normal abelian subgroup $A\subset H$ such that the quotient group $H/A$ is abelian.)

This paper also asks the following question:
\begin{question}[Banakh]\label{que:banakh}
Is each countable subset of an uncountable Polish group $G$ contained in a $G_\delta$ Haar null subset of $G$?
\end{question}
The following well-known and seemingly unrelated problem was asked in \cite{DaOP}:
\begin{question}[Darji]\label{que:darji}
Can every uncountable Polish group be written as the union of two sets, one meager and the other Haar null?
\end{question}

However, it is not very hard to prove that these questions are equivalent:
\begin{fact}[Elekes-Nagy]\label{fact:equivques}
If $G$ is an uncountable Polish group, then the following are equivalent:
\begin{enumerate}[label=(\roman*)]
\item Each countable subset of $G$ is contained in a $G_\delta$ Haar null subset of $G$.
\item $G$ can be written as the union of two sets, one meager and the other Haar null.
\end{enumerate}
\end{fact}

Several papers study \autoref{que:darji} and give affirmative answers in various groups or classes of groups, hence yielding answers for \autoref{que:banakh}. For the proof of \autoref{fact:equivques} and a list of groups where \autoref{que:darji} is answered, see \cite[subsection 5.4]{EN}.

Another related problem concerns the \emph{Haar meager sets} introduced by Darji in \cite{Da}:

\begin{definition}
If $G$ is a Polish group, a set $A\subseteq G$ is called \emph{Haar meager} if there exist a Borel set $B\supseteq A$, a compact metric space $K$ and a continuous function $f: K\to G$ such that $f^{-1}(gBh)$ is meager in $K$ for every $g, h\in G$. If here $K\subseteq G$ and $f$ is the identity function on $K$, then we say that $A$ is \emph{strongly Haar meager}.
\end{definition}

The Haar meager sets form a translation-invariant $\sigma$-ideal and coincide with meager sets in locally compact Polish groups; for more information about them, see \cite{Da}, \cite{DRVV} or the survey paper \cite{EN}.

It was proved in \cite{DV} that the analog of \autoref{the:ev} holds for Haar meager sets:
\begin{theorem}(Dole\v{z}al-Vlas\'ak)
Let $G$ be a non-locally compact abelian Polish group and $1\le \xi\le\omega_1$. Then there is a strongly Haar meager Borel set without any Haar meager hull in $\mathbf{\Sigma}^0_\xi$.
\end{theorem}
(Note that the measure $\leftrightarrow$ category duality usually involves a $\mathbf{\Sigma}\leftrightarrow \mathbf{\Pi}$ swap; therefore it is not surprising that this theorem considers $\mathbf{\Sigma}^0_\xi$ hulls instead of $\mathbf{\Pi}^0_\xi$ hulls.)

It is natural to ask whether our results have an analog for Haar meager sets:
\begin{question}
For a given Polish group $G$ and $1\le \xi<\omega_1$, what is the least complexity of a Haar meager set in $G$ that cannot be covered by a $\mathbf{\Sigma}^0_\xi$ Haar meager set?
\end{question}

\end{document}